\documentclass{amsart}
\usepackage{amsmath,amssymb,graphicx,setspace,verbatim}
\usepackage{color, tikz}
\usetikzlibrary{graphs}
\usetikzlibrary{graphs.standard}
\usepackage{comment}
\usepackage{appendix}
\theoremstyle{comment}

\input xy
\xyoption{all}

\newtheorem{theorem}{Theorem}[section]

\newtheorem{construction}[theorem]{Construction}
\newtheorem{lemma}[theorem]{Lemma}
\newtheorem{corollary}[theorem]{Corollary}

\begin{document}

\title{An infinite family of locally X graphs\\ based on incidence geometries}

\author{Natalia Garcia-Colin}
\address{Natalia Garcia-Colin, CONACYT-INFOTEC. Circuito Tecnopolo Sur No 112, Col. Fracc. Tecnopolo Pocitos C.P. 20313, Aguascalientes, Ags. México.}
\email{natalia.garcia@infotec.mx}

\author{Dimitri Leemans}
\address{Dimitri Leemans, Universit\'e Libre de Bruxelles, D\'epartement de Math\'ematique, C.P.216 - Alg\`ebre et Combinatoire, Boulevard du Triomphe, 1050 Brussels, Belgium}
\email{dleemans@ulb.ac.be}

\begin{abstract}
A graph ${\mathcal G}$ is {\em locally X} if the graphs induced on the neighbours of every vertex of ${\mathcal G}$ are isomorphic to the graph $X$. 
We prove that the infinite family of incidence graphs of the $r$--rank incidence geometries, $\Gamma(KG(n,k),r)$, constructed using the Kneser graphs $KG(n,k)$, are locally $X$ with $X$ being the incidence graphs of the rank $r-1$ residues of $\Gamma(KG(n,k),r)$.
\end{abstract}
\maketitle
\noindent \textbf{Keywords:} Kneser graph, locally X graph, incidence geometry.

\noindent \textbf{2010 Math Subj. Class:} 05C75, 51E24

\section{Introduction}

For a given graph $X$, a graph ${\mathcal G}$ is {\em locally X} if the graphs induced on the neighbours of every vertex of ${\mathcal G}$ are isomorphic to $X$. In the literature, ${\mathcal G}$ can also be referred to as an {\em extension of X} or a {\em locally homogeneous graph.}

A. Zykov \cite{Zykov,Zyk63} posed the problem of characterizing the graphs, $X,$ for which there are locally $X$ graphs. Finding any general solution for this problem is difficult, if at all possible. 

Apart from the inherent interest of this problem for graph theorists, another motivation for the study of locally homogeneous graphs is observed in \cite{Hal87}:

{\em ``The theorems may find application in the characterization of the Johnson scheme among the primitive association schemes and distance regular graphs. It can also be used to characterize alternating and symmetric groups (of sufficiently large degree) by centralizers of various of their elements (the initial motivation for the theorem)."}

The progress thus far has followed three general lines of enquiry; the undecidability of the problem; the construction of locally $X$ graphs for some selected graphs, $X$; and sufficient conditions for a graph $X$ to have an extension.

It has been shown that many extension problems are undecidable \cite{Bul72,Bul73}. Among them, it is algorithmically undecidable to determine whether a given graph $L$ has a finite or infinite extension, and to determine whether a given graph $L$ has an {infinite} extension. However, it is still unknown whether the problem of determining when a given graph $L$ has a finite extension is decidable or undecidable. A survey paper \cite{Hel78} dealing with the progress in solving these problems was published in 1976.

In \cite{Hal85}, a complete list of all graphs $X$ of order up to six having an extension is given; in some cases all such extensions are characterized. Constructions of locally $X$ graphs, for instance, cycles \cite{MR0357220, MR1376183}, unions of paths \cite{Parsons1989}, trees \cite{BHM80}, polyhedra \cite{BUSET1983221, MR826792, MR826793}, the Petersen graph \cite{Hall1980}, other Kneser graphs \cite{Hal87,Moon84}, dense graphs \cite{Bugata1989, Ned92} and others have been investigated in \cite{BC73,BC75,BKK03,MR1084898, BUEKENHOUT1977391, Doyen1976, Hall1985GD, Hell1977}. A rich compilation of locally $X$ graphs can also be found in \cite{MathWorld}.

Some structural characteristics for a graph $X$ to have an extension have been given in \cite{Al10,BC73,BC75,Hel78,Wee94}. For example, in \cite{Wee94} it is proved that if $X$ is a regular graph of degree greater than one and girth at least six then $X$ has an extension. In \cite{Vog84}, its proven that if there exists a finite locally $X$ graph ${\mathcal G}$, where $X$ is a disconnected finite graph, then for an arbitrary finite group $\Gamma$ there are infinitely many connected locally $X$ graphs such that the automorphism group $Aut(\mathcal G) \cong \Gamma$.

Finally, we may mention the problem of local homogeneity of graphs is related to algebraic topology and to group theory~\cite{Vin81, Ronan81, SUROWSKI1985371}.

The paper is organized as follows. In Section \ref{back} we present some structural properties of Kneser graphs which will be used in the construction of the incidence geometries we study and we present the necessary background on incidence geometries. In Section \ref{sec:Kneser} we compute the neighbourhood geometry of a Kneser graph, which is used in Section \ref{sec:locallyX} to construct a new infinite family of locally $X$ graphs.
 
\section{background}\label{back}

\subsection{Graph theory}



A {\em Kneser graph\footnote{Lov\'asz introduced the term Kneser graph in~\cite{Lovasz} after a problem posed by Kneser in~\cite{Kneser}.}} $KG(n,k)$ is a graph whose vertex set is the set of all $k$-subsets of $\{1, \ldots, n\}$ and any pair of disjoint subsets is joined by an edge.
Complete graphs are very familiar objects, for instance $K_n$ are Kneser graphs $KG(n,1)$ and the Petersen graph is a $KG(5,2)$. Its very simple to see that Kneser graphs are locally Kneser graphs, furthermore Jonathan Hall proved in~\cite{Hall1980} that there are exactly three pairwise non-isomorphic locally Petersen graphs, only one of them being a Kneser graph, namely $KG(7,2)$.

The following lemmas will cover some structural characteristics of Kneser graphs which, in turn, will be used in the later sections for determining structural characteristics of our constructions.

\begin{lemma}\label{lem:oddcycle}
The smallest odd cycle of a Kneser graph, $KG(n,k)$, with $n>2k+1$ has length $2\lceil\frac{k}{n-2k}\rceil +1.$
\end{lemma}

\begin{proof}

We may assume that $n<3k$, as $KG(n,k)$ with $n\geq 3k$ has triangles, and the statement holds. Let $n=2k+r$ for some $r<k$ and $A_1, \ldots, A_{2l+1}$ be the vertices of the smallest odd cycle, as subsets of $[n]$.

By construction, we have $A_1 \cap A_2= \emptyset$ and $A_2 \cap A_3= \emptyset$ thus $A_1 \cup A_3 \subset A_2 ^c$ and they have non empty intersection as,  $|A_2 ^c|=n-k=k+r<2k$, thus $|A_1 \cap A_3|\geq k-r$ and $A_1$ and $A_3$ cannot be adjacent. Similarly, we can argue that $|A_3 \cap A_5|\geq k-r$ thus $|A_1 \cap A_5|\geq k-2r$. We may continue this process to conclude that $|A_1 \cap A_{2i+1}|\geq k-ir$. 

Hence $A_1 \cap A_{2l+1} = \emptyset$ if and only if $k-lr\leq 0$. This happens precisely when $\lceil\frac{k}{n-2k}\rceil\leq l$ and the result follows.
\end{proof}

\begin{lemma}\label{lem:paths}
Between any two vertices of a Kneser graph, $A,B$, such that $|A \cap B|=c$ there is an even path of length  $2\lceil\frac{k-c}{n-2k}\rceil$ and an odd path of length $2\lceil\frac{c}{n-2k}\rceil+1.$
\end{lemma}

\begin{proof}
Let $A$ and $B$ be two vertices of $KG(n,k)$, $C=A \cap B$, $|C|=c$,  $D=(A \cup B)^c$, and $|D|=n-2k+c.$ Let $X_1, \ldots , X_{l}$ and $Y_1, \ldots, Y_{l}$ with $l=\lceil\frac{k-s}{n-2k}\rceil$ be partitions of $A \setminus B$ and $B\setminus A$, respectively, in sets of size $n-2k$, perhaps except the last one.

Let $A_{2i-1}= (\cup_{j=1}^{i}X_j)\cup (\cup_{j=i+1}^{l}Y_j) \cup D'$ for some $D'\subset D$ of cardinality $c$ and $A_{2i}= (\cup_{j=1}^{i} Y_j)\cup (\cup_{j=i+1}^{l} X_j) \cup C$ for $0\leq i\leq l.$ Clearly $A_j \cap A_{j+1}= \emptyset$, $A = A_0$, $B =A_{2l}$, thus $A, A_1 \ldots, A_{2l-1}, B$ is a path of even length $2\lceil\frac{k-c}{n-2k}\rceil$.

For the odd path, take $D'\subset D$ of cardinality $c$ and let $A'=B\setminus A \cup D'$. Then $A$ and $A'$ are adjacent, and we may construct an even path of size $2\lceil\frac{c}{n-2k}\rceil$ between $A'$ and $B$ as before, given that $|A' \cap B|=k-c,$ and the result follows. 
\end{proof}


\subsection{Incidence geometry}

An {\em incidence system} \cite{BuekCohen}, $\Gamma := (X, *, t, I)$ is a 4-tuple such that
\begin{itemize}
\item $X$ is a set whose elements are called the {\em elements} of $\Gamma$;
\item $I$ is a set whose elements are called the {\em types} of $\Gamma$;
\item $t:X\rightarrow I$ is a {\em type function}, associating to each element $x\in X$ of $\Gamma$ a type $t(x)\in I$;
\item $*$ is a binary relation on $X$ called {\em incidence}, that is reflexive, symmetric and such that for all $x,y\in X$, if $x*y$ and $t(x) = t(y)$ then $x=y$.
\end{itemize}
The {\em incidence graph} of $\Gamma$ is the graph whose vertex set is $X$ and where two vertices are joined provided the corresponding elements of $\Gamma$ are incident. 

A {\em flag} is a set of pairwise incident elements of $\Gamma$, i.e. a clique of its incidence graph.
The {\em type} of a flag $F$ is $\{t(x) : x \in F\}$.
A {\em chamber} is a flag of type $I$.
An element $x$ is {\em incident} to a flag $F$ and we write $x*F$ for that, when $x$ is incident to all elements of $F$.
An incidence system $\Gamma$ is a {\em geometry} or {\em incidence geometry} if every flag of $\Gamma$ is contained in a chamber (or in other words, every maximal clique of the incidence graph is a chamber).
The {\em rank} of $\Gamma$ is the number of types of $\Gamma$, namely the cardinality of $I$.

Observe that the incidence graph of an incidence system of rank $n$ is an $n$-partite graph. This will play a key role in the construction of an infinite family of locally X graphs.

Let $\Gamma:= (X, *, t, I)$ be an incidence system.
Given $J\subseteq I$, the {\em $J$--truncation} of $\Gamma$ is the incidence system $\Gamma^J := (t^{-1}(J), *_{\mid t^{-1}(J)\times t^{-1}(J)}, t_{\mid J}, J)$. In other words, it is the subgeometry constructed from $\Gamma$ by taking only elements of type $J$ and restricting the type function and incidence relation to these elements.

Let $\Gamma:= (X, *, t, I)$ be an incidence system.
Given a flag $F$ of $\Gamma$, the {\em residue} of $F$ in $\Gamma$ is the incidence system $\Gamma_F := (X_F, *_F, t_F, I_F)$ where
\begin{itemize}
\item $X_F := \{ x \in X : x * F, x \not\in F\}$;
\item $I_F := I \setminus t(F)$;
\item $t_F$ and $*_F$ are the restrictions of $t$ and $*$ to $X_F$ and $I_F$.
\end{itemize}

An incidence system $\Gamma$ is {\em residually connected} when each residue of rank at least two of $\Gamma$ has a connected incidence graph. It is called {\em firm} (resp. {\em thick}) when every residue of rank one of $\Gamma$ contains at least two (resp. three) elements.  

Let $\Gamma:=(X,*, t,I)$ be an incidence system.
An {\em automorphism} of $\Gamma$ is a mapping
$\alpha:(X,I)\rightarrow (X,I):(x,t(x)) \mapsto (\alpha(x),t(\alpha(x))$
where
\begin{itemize}
\item $\alpha$ is a bijection on $X$ inducing a bijection on $I$;
\item for each $x$, $y\in X$, $x*y$ if and only if $\alpha(x)*\alpha(y)$;
\item for each $x$, $y\in X$, $t(x)=t(y)$ if and only if $t(\alpha(x))=t(\alpha(y))$.
\end{itemize}
An automorphism $\alpha$ of $\Gamma$ is called {\em type preserving} when for each $x\in X$, $t(\alpha(x))=t(x)$.
The set of all automorphisms of $\Gamma$ together with the composition forms a group that is called the {\em automorphism group} of $\Gamma$ and denoted by $Aut(\Gamma)$.
The set of all type-preserving automorphisms of $\Gamma$ is a subgroup of $Aut(\Gamma)$ that we denote by $Aut_I(\Gamma)$.
An incidence system $\Gamma$ is {\em flag-transitive} if $Aut_I(\Gamma)$ is transitive on all flags of a given type $J$ for each type $J \subseteq I$.

A rank two geometry with points and lines is called a {\em generalised digon} if every point is incident to every line and, it is called a {\em partial linear space} if there is at most one line through any pair of points.
An incidence geometry is said to satisfy the {\em intersection property of rank 2}, denoted by $(IP)_2$, when all its rank two residues are either partial linear spaces or generalised digons.

Let $\Gamma$ be a firm, residually connected and flag-transitive geometry.  The \emph{Buekenhout diagram} of $\Gamma$ is a graph whose vertices are the elements of $I$ and with 
an edge $\{i,j\}$ with label $d_{ij} - g_{ij} - d_{ji}$ whenever every residue of type $\{i,j\}$ is not a generalised digon. The number $g_{ij}$ is called the {\em gonality} and is equal to half the girth of the incidence graph of a residue of type $\{i,j\}$. The number $d_{ij}$ is called the $i$-diameter of a residue of type $\{i,j\}$ and is the longest distance from an element of type $i$ to any element in the incidence graph of the residue.
Moreover, to every vertex $i$ is associated a number $s_i$, called the $i$-order, which is equal to the size of a residue of type $i$ minus one, and a number $n_i$ which is the number of elements of type $i$ of the geometry.

The Petersen graph, for instance, can be seen as a geometry of rank two whose elements are the vertices and edges of the graph. Its Buekenhout diagram is the following.

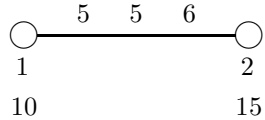
\begin{center}
\begin{picture}(90,60)
\put(0,40){\circle{10}}
\put(85,40){\circle{10}}
\put(5,40){\line(1,0){75}}
\put(20,45){$5$}
\put(40,45){$5$}
\put(60,45){$6$}
\put(-3,25){1}
\put(-5,10){10}
\put(82,25){2}
\put(80,10){15}
\end{picture}
\end{center}

Let $G$ be a graph. Denote its set of vertices (resp. edges) by $G_0$ (resp. $G_1$). For distinct $p$, $q \in G_0$, we say that $p$ and $q$ are {\em adjacent} -- and we write $p \sim q$ -- whenever $\{p, q\} \in G_1$.
As in~\cite{LPL2000}, to the graph $G$, we associate a new rank 2 geometry $\tilde G$, called the {\em neighborhood geometry} of $G$, whose elements are, roughly speaking, the vertices and the neighborhoods of vertices of $G$. 
More precisely, we define $\tilde G$ to be the geometry $(G_0 \times \{0\} \cup G_0 \times \{1\}, \tilde *, \tilde t, \{0,1\})$ with
\begin{itemize}
    \item $\tilde t(G_0\times \{i\}) = i$, for $i = 0, 1$;
    \item $(p,0) \tilde * (q,1)$ iff $p \sim q$, for $p, q \in G_0$.
\end{itemize}

As pointed out in~\cite[Table 1]{LPL2000}, the neighborhood  geometry of the Petersen graph is Desargues' configuration.
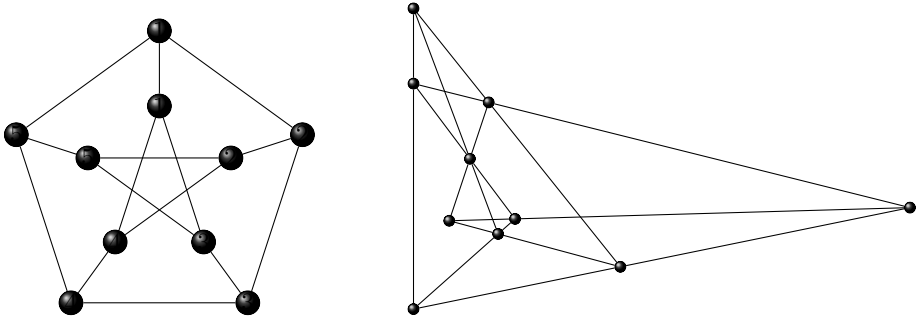
\begin{figure}
\begin{tikzpicture}[every node/.style={ball color=black, circle, draw=black, inner sep=0.01cm}]
  \graph[clockwise, radius=2cm] {subgraph C_n [n=5,name=A]};
  \graph[clockwise, radius=1cm] {subgraph I_n [n=5,name=B]};

  \foreach \i in {1,2,3,4,5}{\draw (A \i) -- (B \i);}
  \newcounter{j}
  \foreach \i in {1,2,3,4,5}{%
  \pgfmathsetcounter{j}{ifthenelse(mod(\i+2,5),mod(\i+2,5),5)}
  \draw (B \i) -- (B \thej);
  }
\end{tikzpicture}
\hspace{1cm}
\begin{tikzpicture}
\tikzstyle{point1}=[ball color=black, circle, draw=black, inner sep=0.05cm]
\node (v1) at (0,4) [point1] {};
\node (v2) at (0,3) [point1] {};
\node (v3) at (1,2.75) [point1] {};
\node (v4) at (0.75,2) [point1] {};
\node (v5) at (0,0) [point1] {};
\node (v6) at (2.75,4-2.75*1.25) [point1] {};
\node (v7) at (1.5*1.5/2,4-1.5*2) [point1] {};
\draw (v1) -- (v2) -- (v5);
\draw (v1) -- (v3) -- (v6);
\draw (v1) -- (v4) -- (v7);
\draw (v2) -- (v3) -- (v4) -- (v2);
\draw (v5) -- (v6) -- (v7) -- (v5);
\node (v8) at (intersection of v2--v3 and v5--v6) [point1] {};
\node (v9) at (intersection of v2--v4 and v5--v7) [point1] {};
\node (v10) at (intersection of v3--v4 and v6--v7) [point1] {};
\draw (v3) -- (v8) -- (v6);
\draw (v4) -- (v9) -- (v7);
\draw (v4) -- (v10) -- (v7);
\draw (v8) -- (v9) -- (v10);
\end{tikzpicture}
\caption{Petersen graph $KG(5,2)$ and Desargues' configuration $\tilde{KG}(5,2)$}\label{PetDes}
\end{figure}
Figure~\ref{PetDes} gives the Petersen graph and Desargues' configuration.
The Buekenhout diagram of Desargues' configuration is the following.
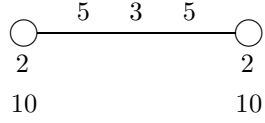
\begin{center}
\begin{picture}(90,60)
\put(0,40){\circle{10}}
\put(85,40){\circle{10}}
\put(5,40){\line(1,0){75}}
\put(20,45){$5$}
\put(40,45){$3$}
\put(60,45){$5$}
\put(-3,25){2}
\put(-5,10){10}
\put(82,25){2}
\put(80,10){10}
\end{picture}
\end{center}

\section{The neighborhood geometry of a Kneser graph} \label{sec:Kneser}

In this section, we compute the neighborhood geometry of a given Kneser graph. These geometries will then be used in the next section to construct locally X graphs as incidence graphs of some particular incidence geometries. The incidence graphs of these geometries are sometimes called the bipartite Kneser graphs.



\begin{lemma}\label{lem:diameter} The $0$--diameter and $1$--diameter of $\tilde{KG}(n,k)$ is $2 \lceil\frac{k}{n-2k}\rceil +1.$
\end{lemma}

\begin{proof}
As the construction of $\tilde{KG}(n,k)$ is symmetric in the set of types, the $0$--diameter and the $1$--diameter are the same.

By Lemma \ref{lem:paths} the distance between any two vertices of $\tilde{KG}(n,k)$ is at most $2\lceil\frac{k-c}{n-2k}\rceil$ or $2\lceil\frac{c}{n-2k}\rceil+1,$ for some $0\leq c \leq k.$ This achieves a maximum of $2 \lceil\frac{k}{n-2k}\rceil +1$ when $c=k$, that is, when we are tracing a path from the two copies of the same vertex in $\tilde{KG}(n,k)$.

We now argue that this bound can't be improved, as such improvement would contradict Lemma \ref{lem:oddcycle}.
\end{proof}

\begin{lemma}\label{lem:gonality}
The gonality of $\tilde{KG}(n,k)$ is $3$ when $n=2k+1$ and $2$ when $n\geq 2k+2$.
\end{lemma}
\begin{proof}
If $n=2k+1$, the following is a circuit of length 6 in the incidence graph of $\tilde{KG}(n,k)$: $(\{1,\ldots, k\},0)
\tilde{*}(\{k+1,\ldots, 2k\},1)\tilde{*}(\{1,\ldots, k-1,2k+1\},0)\tilde{*}(\{k,\ldots, 2k-1\},1)\tilde{*}(\{1,\ldots, k-1,2k\},0)\tilde{*}(\{k+1,\ldots, 2k-1,2k+1\},1)\tilde{*}(\{1,\ldots, k\},0)$.

If $n>2k+1$, the following is a circuit of length 4 in the incidence graph of $\tilde{KG}(n,k)$: $(\{1,\ldots, k\},0)
\tilde{*}(\{k+1,\ldots, 2k\},1)\tilde{*}(\{1,\ldots, k-1,2k+1\},0)\tilde{*}(\{k+1,\ldots, 2k-1,2k+2\},1)\tilde{*}(\{1,\ldots, k\},0)$.
\end{proof}

\begin{lemma}\label{lem:connected}
$\tilde{KG}(n,k)$ is a connected graph.
\end{lemma}
\begin{proof}
This follows from the fact that $KG(n,k)$ has cycles of odd length as proven in Lemma~\ref{lem:oddcycle}.
\end{proof}
\section{A new family of locally X graphs}\label{sec:locallyX}

Take a Kneser graph $KG(n,k)$. Its neighborhood geometry $\tilde{KG}(n,k)$ can be used to construct infinitely many locally X graphs using incidence geometry.

\begin{construction}\label{construction}
For any positive integer $r\geq 2$ and a given Kneser graph $KG(n,k)$, define a rank $r$ incidence system $\Gamma(KG(n,k),r) := (X,*,t,I)$ as follows.
Let $\Omega := \{1, \ldots, n+k(r-2)\}$. Let $I:=\{1,\ldots, r\}$.
Take $r$ copies $X_1, \ldots, X_r$ of all the subsets of size $k$ of $\Omega$, and let $X=X_1\cup \ldots \cup X_r$.
For any $x\in X_i$, define $t(x) = i$.
For any elements $x_i\in X_i$ and $x_j\in X_j$, we say that $x_i*x_j$ if and only if $i\neq j$ and $x_i$ and $x_j$ are disjoint as subsets of $\Omega$.
\end{construction}

Observe that the case $r=2$ in Construction~\ref{construction} gives the neighbourhood geometry $\tilde{KG}(n,k)$ of the Kneser graph $KG(n,k)$ that was defined in the previous section.

\begin{lemma}
$\Gamma(KG(n,k),r)$ is an incidence geometry.
\end{lemma}
\begin{proof}
As $\mid \Omega \mid = n+k(r-2)$ and $n\geq 2k+1$, it is always possible to find $r$ pairwise disjoint subsets of size $k$ in $\Omega$. Hence every maximal flag of $\Gamma$ must be a chamber and $\Gamma$ is an incidence geometry.
\end{proof}
\begin{lemma}\label{lem:ft}
The symmetric group $S_\Omega \cong S_{n+k(r-2)}$ acts transitively on the chambers of $\Gamma(KG(n,k),r)$ (or in other words, $\Gamma(KG(n,k),r)$ is flag-transitive).
\end{lemma}
\begin{proof}
This is due to the fact that $S_\Omega$ is $(n+k(r-2))$--transitive on $\Omega$.
\end{proof}
\begin{lemma}
$\Gamma(KG(n,k),r)$ is residually connected.
\end{lemma}
\begin{proof}
We prove this by induction on $r$. 
The case $r=2$ is dealt with in Lemma~\ref{lem:connected}.

Suppose that $\Gamma(KG(n,k),r)$ is residually connected.
In order to prove that $\Gamma(KG(n,k),r+1)$ is residually connected, we only need to show that the incidence graph of $\Gamma(KG(n,k),r+1)$ is connected, as all residues of $\Gamma(KG(n,k),r+1)$ of rank $< r+1$ are connected by the induction hypothesis and the fact that $\Gamma(KG(n,k),r+1)$ is flag-transitive as shown in Lemma~\ref{lem:ft}. Take $x_i$ an element of type $i$ and $x_j$ an element of type $j\neq i$. The $\{i,j\}$--truncation of $\Gamma(KG(n,k),r)$ is the neighborhood geometry $\tilde{KG}(n*k(r-2),k)$ of a Kneser graph $KG(n*k(r-2),k)$. By Lemma~\ref{lem:connected}, $\tilde{KG}(n*k(r-2),k)$ is connected. Hence, every rank two truncation of $\Gamma(KG(n,k),r)$ is connected and therefore $\Gamma(KG(n,k),r)$ is connected.
\end{proof}

\begin{lemma}\label{lem:order}
For any $i = 1, \ldots, r$, the $i$--order of $\Gamma(KG(n,k),r)$ is equal to $n-k \choose k$$-1$. 
\end{lemma}
\begin{proof}
A flag $F$ of rank $r-1$ consists of $r-1$ pairwise disjoint subsets of size $k$. Hence these subsets cover $k(r-1)$ points of $\Omega$. So there are $n+k(r-2)-k(r-1) = n-k$ points not covered by $F$ in $\Omega$. There are thus $n-k \choose k$ subsets of size $k$ that are disjoint with all subsets of $F$.
\end{proof}
The previous lemma immediately implies the following corollary.
\begin{corollary}
$\Gamma(KG(n,k),r)$ is thick.
\end{corollary}

\begin{lemma}
Every rank two residue of $\Gamma(KG(n,k),r)$ is isomorphic to $\tilde{KG}(n,k)$.
\end{lemma}
\begin{proof}
Every rank two residue is obtained from a flag $F$ that has $r-2$ elements. These $r-2$ elements cover $k(r-2)$ elements of $\Omega$ and therefore there are $n$ elements of $\Omega$ remaining.
\end{proof}

\begin{lemma}
$Aut_I(\Gamma(KG(n,k),r))\cong S_\Omega$ and $Aut(\Gamma(KG(n,k),r))\cong S_\Omega\times S_r$
\end{lemma}
\begin{proof}
By Lemma~\ref{lem:ft}, we know that $Aut_I(\Gamma(KG(n,k),r))\geq S_\Omega$. Obviously, it cannot be strictly bigger.
Construction~\ref{construction} is symmetric in the set of types. Hence $Aut(\Gamma(KG(n,k),r)\cong S_\Omega\times S_r$.
\end{proof}

We summarize in the following theorem all the properties we have proved on $\Gamma(KG(n,k),r)$. Note that the gonality and diameters of the rank two residues of the Buekenhout diagram were computed in Lemma~\ref{lem:gonality} and Lemma~\ref{lem:diameter}.   

\begin{theorem}
$\Gamma(KG(n,k),r)$ is a thick, residually connected and flag-transitive incidence geometry of rank $r$. Its automorphism groups $Aut_I(\Gamma) \cong S_\Omega$ and $Aut(\Gamma) \cong S_\Omega\times S_r$. The Buekenhout diagram of $\Gamma(KG(n,k),r)$ is a complete graph. The orders  of the diagram are $n-k \choose k-1$, the number of elements of each type is ${n+k(r-2)} \choose k$, the edges are labelled as $d-g-d$ where $g=3$ if $n=2k+1$ and $2$ otherwise, and $d=2\lceil\frac{k}{n-2k}\rceil +1$.
\end{theorem}

\begin{corollary}
The incidence graph of $\Gamma(KG(n,k),r)$ is a locally X graph.
\end{corollary}
\begin{proof}
Since $\Gamma(KG(n,k),r)$ is a flag-transitive geometry and since $Aut(\Gamma(KG(n,k),r)\cong S_\Omega\times S_r$, all the incidence graphs of the residues of rank $r-1$ are isomorphic.
\end{proof}
\begin{corollary}
There exists a locally Desargues graph whose automorphism group is isomorphic to $S_7$.
\end{corollary}
\begin{proof}
Such a graph can be obtained as the incidence graph of $\Gamma(KG(5,2),3)$.
\end{proof}

We highlight that $\Gamma(KG(5,2),3)$ was already given in~\cite[page 86]{BDL1996}. Furthermore, geometries satisfying the intersection property of rank two, $(IP)_2$, attracted much attention in the nineties and noughties (see, for instance,~\cite{BDL1996,Lee2008}). It turns out that some of the incidence geometries obtained by Construction~\ref{construction} satisfy $(IP)_2$, as shown in the next corollary.

\begin{corollary}
$\Gamma(KG(n,k),r)$ is $(IP)_2$ if and only if $n=2k+1$.
\end{corollary}
\begin{proof}
For a rank two residue of $\Gamma(KG(n,k),r)$ to satisfy $(IP)_2$, we need it to be a generalised digon or its gonality to be at least 3. Generalised digons have gonality and diameters equal to two. Lemma~\ref{lem:diameter} and Lemma~\ref{lem:gonality} then finish the proof.
\end{proof}

Observe that when $n=2k+1$, the diameters and gonality written on the edges of the Buekenhout diagram are respectively $n$ and $3$.

\bibliographystyle{plain}

\end{document}